\documentclass{amsart}
\usepackage{amscd,amssymb}

\theoremstyle{plain}
\newtheorem{prop}{Proposition}[section]
\newtheorem{thm}[prop]{Theorem}

\newtheorem{remark}[prop]{Remark}
\newtheorem{lemma}[prop]{Lemma}

\theoremstyle{definition}

\theoremstyle{remark}

\numberwithin{equation}{section}
\numberwithin{subsection}{section} %\numberwithin{table}{section}

\newcommand{\N}{\mathbb{N}}

\newcommand{\f}[1]{\mbox{${\mathbb F}_{q^{#1}}$}}

\newcommand{\vertex}{\hbox{\raise2.5pt\vbox{\hrule width 2truecm}}}
\newcommand{\svertex}{\hbox{\raise2.5pt\vbox{\hrule width .3truecm}}}
\newcommand{\sdvertex}{\hbox{\lower1.5pt
\vbox{\offinterlineskip\svertex\svertex}}}
\newcommand{\dvertex}{\hbox{\lower1.5pt
\vbox{\offinterlineskip\vertex\vertex}}}

\newcommand{\argu}{\hbox to 7truept{\hrulefill}}

\begin{document}
\title{$2$-coverings of classical groups}
%including the case $Sp_n(q)$ with $n=3,\ 4$
 %and the unitary cases $SU_n(q)$ with $n=3,\ 4,\ 6.$
\author{D. Bubboloni \and M. S. Lucido \and Th. Weigel}
\address{D. Bubboloni\\
Dipartimento di Matematica per le Decisioni, Universit\`a degli Studi di Firenze\\ Via C. Lombroso 6/17\\
I-50134 Firenze, Italy} \email{daniela.bubboloni@dmd.unifi.it}
\address{M.S.Lucido\\
Dipartimento di Matematica ed Informatica\\ Universit\`a degli Studi di Udine\\
Via delle Scienze 206\\ I-33100 Udine, Italy}
\email{mariasilvia.lucido@dimi.uniud.it}
\address{Th.Weigel\\
Dipartimento di Matematica ed Applicazioni\\ Universit\`a degli
Studi Milano-Bicoc\-ca\\ U5-3067, Via R.Cozzi, 53\\ I-20125
Milano, Italy} \email{thomas.weigel@unimib.it}

\begin{abstract}In this paper we show that if $n\geq 5$ and $G$
is any of the groups $SU_n(q)$ with $n\neq 6,$  $Sp_{2n}(q)$ with
$q$ odd, $\Omega_{2n+1}(q),$ $\Omega_{2n}^{\pm}(q),$ then $G$ and
the simple group $\overline G=G/Z(G)$ are not 2-coverable. Moreover
the only 2-covering of $Sp_{2n}(q),$ with $q$ even, has components $
O^-_{2n}(q)$  and $O^{+}_{2n}(q) .$

\end{abstract}
%\begin{document}
\maketitle

\section{Introduction}
\label{s:intro}

It is well known that a finite group $G$ is never the
set-theoretical union of the $G$-conjugates of a proper subgroup.
However there are examples of groups which are the set-theoretical
union of the $G$-conjugates of two proper subgroups.\\ Let $G$ be a
group and let $H,\ K$  be proper subgroups of $G.$ If every element
of $G$ is $G$-conjugate to an element of $H$, or to an element of
$K$ i.e. $$G=\bigcup_{g\in G} H^g \cup \bigcup_{g\in G} K^g,$$ then
$\delta=\{H,\ K\}$ is called a {\it $2$-covering} of $G$, with {\it
components} $H,\ K$ and $G$ is said {\it $2$-coverable}. For any
$H<G,$ we use the notation $[H]$ for $\bigcup_{g\in G} H^g.$ Observe
that we can assume $H,\ K$ maximal subgroups of $G.$ In what follows
the components of a $2$-covering will be tacitly assumed as maximal
subgroups of $G.$ If $H$ is a maximal subgroup of $G,$ we write $H<
\cdot\, G.$
 \\ %For any $H\leq G,$ we
%will use the notation $[H]$ to denote the set $\bigcup_{g\in G}
%H^g.$\\In \cite {b} it has been proved that $Alt(n)$ is
%$2$-coverable if and only if  $\;4\leq n\leq 8$ and  $Sym(n)$ is
%$2$-coverable if and only if $\,3\leq n\leq 6\,$  and in
%\cite{bbh} it was explored an application of these results to a
%question of number theory.

In \cite{bl}, it has been proved that the  linear groups
$GL_{n}(q),$ $SL_{n}(q)$ and the projective groups  $PGL_{n}(q)$,
$PSL_{n}(q)$ are $2$-coverable if and only if $2 \le n \le 4$. Here
we consider the other classical groups. Our Main Theorem is a
collection of the statements
 in the sections 3-7 (see Propositions \ref{simplettici},
\ref{unitari},  \ref{ortogonali dispari}, \ref{ortogonali +},
\ref{ortogonali-}) and of
 Remark \ref{centre}.\\

{ \bf Main Theorem} {\em  Let $n\geq 5$ and $G$ be any of the groups
$SU_n(q)$ with $n\neq 6,$ $Sp_{2n}(q)$ with $q$ odd,
$\Omega_{2n}^{+}(q)$, $\Omega_{2n+1}(q),$ $\Omega_{2n}^{-}(q).$ Then
$G$ and the simple group $\overline G=G/Z(G)$ are not 2-coverable.
If $q$ is even the group $Sp_{2n}(q)$ has only a 2-covering which
has components isomorphic to $ O^-_{2n}(q)$  and $O^{+}_{2n}(q) .$}\\
R.H. Dye in \cite{Dye} proved that in fact $\,\{O^-_{2n}(q),\
O^{+}_{2n}(q)\}\,$ is a $2$-covering for $Sp_{2n}(q),$ with $q$
even.

%%%%%%%%%%%%%%%%%%%%%%%%%%%%%%%%%%%%%%%%%%%%%%%%%
\section{Preliminary facts}
Let $q=p^f$ be a prime power and $n \ge 3$. We consider the
following {\em classical groups} groups $G$: $$SU_n(q) \ \hbox{ with
}\ q=q_0^2,\ \quad Sp_{2n}(q),\quad \Omega_{2n+1}(q)\ \quad
\Omega_{2n}^{\pm}(q).$$ The corresponding {\em general classical
groups} $\tilde G$ are:
$$GU_n(q),\ \quad Sp_{2n}(q),\quad O_{2n+1}(q),\
\quad O_{2n}^{\pm}(q).$$ Observe that $\tilde {G}'=G.$\\

Let $V$ be the natural $\f{}\tilde{ G}$-module endowed with the
suitable non degenerate form and put $d=dim_{_{\,\f{}}}V.$
Sometimes, when we need to put in evidence the dimension $d$ and
the field $\f{},$ we will use the notation $G_d(q)$ instead of
$G.$
 If the action of $\,\langle g \rangle
\leq \tilde {G}\,$ decomposes $V$ into the direct sum of
irreducible submodules $V_i$ of dimensions $d_i,\ i=1,\dots,k\,$
we shortly say that {\em the action of $g\in \tilde {G}$ is of
type $d_1\oplus\cdots\oplus d_k.$} Note that $g$ operates irreducibly
 if and only if its characteristic polynomial is irreducible.\\
Since we want to decide when there exist maximal subgroups $H,\ K$
of $G$ such that any element of $G$ belongs to a conjugate  of $H$
or $K$ in $G,$ it is natural to adopt the systematic description of
the maximal subgroups of the classical groups given by M. Aschbacher
in \cite{a}. There, several families $\mathcal{C}_i,\ i=1,\dots,8$
of subgroups  were defined in terms of the geometric properties of
their action on the underlying vector space $V$ and the main result
states that any maximal subgroup belongs to
$\bigcup_{i=1}^{8}\mathcal{C}_i$ or to an additional family
$\mathcal{S}.$ For the notation, the structure theorems on these
maximal subgroups and other details of our investigation we refer to
the book of P. B. Kleidman and M. W. Liebeck (\cite{kl}): in
particular we use the definitions given there of the families
$\mathcal{C}_i.$

We
 consider some special elements to identify the components of a $2$-covering of a classical group.
  First of all we recall the elements
in $\tilde {G}$ or $G$ of maximal order with irreducible action on
$V,$ the so called {\em Singer cycles}.
 In Table \ref{1} we collect the general classical groups $\tilde{G}$
  for which the Singer cycles exist, the order of the Singer
  cycle in $\tilde G$ and in the classical group $G=\tilde{G}'$ (\cite {hu}).

\begin{table}[ht]
\caption{Orders of the Singer cycles}\label{1}
\begin{tabular}{|c|c|c|}
 \hline
 \qquad\  $\widetilde{G}$ \qquad\  &  \qquad\ order
\qquad\
   &  order in $ G=\widetilde{G}'$\   \\
\hline \hline
 $GL_n(q)$ &  $q^n-1$ & $(q^n-1)/(q-1)$  \\
 \hline
 $Sp_{2n}(q)$ &  $q^n+1$ & $q^n+1$ \\
 \hline
 $GU_{n}(q_0^2),$ \hbox{{\em n odd}} &  $q_0^n+1$  & $(q_0^n+1)/(q_0+1)$ \\
 \hline
$O_{2n}^-(q)$ &  $q^n+1$  & $(q^n+1)/(2,q-1)$  \\
 \hline
\end{tabular}
\end{table}
Recall that the Singer cycles are always, up to conjugacy, linear
maps $\pi_a:V\rightarrow V$ of $V=\f{d}$ given by multiplication
$\pi_a(v)=av$ in the field by a suitable $a\in \f{d}^*.$
 To manage expressions of the type $q^a\pm 1,$ it is useful to
state also an easy, technical lemma.
\begin{lemma}\label{aritme} Let $q$ a prime power and $a,\ b,\ n\in
\mathbb{N}.$ Then we have the following:
 \begin{itemize}
\item[ i)] $(q^{a}-1, q^{b}-1)=q^{(a,b)} -1$;
 \item[ ii)] if $a$
is odd then  $(q^{a} +1)/(q+1)$ is odd;
 \item[ iii)] $\bigg
(\frac{q^{a}-1}{q-1}, q-1\bigg )$ divides  $(a, q-1);$
 \item[ iv)]
if  $a$ is odd, then $\bigg (\frac{q^{a}+1}{q+1}, q+1\bigg )$
divides $(a, q+1)$ and $(q^a-1,q+1)=(2,q-1).$
\item[ v)]  if  $a$ is odd and $(a,b)=1$, then
$\bigg (\frac{q^{a}+1}{q+1}, q^{b}+1\bigg )$ divides $(a, q+1);$
 \item[ vi)] if $a$ is odd and $(a,b)=1$, then  $\bigg
(\frac{q^{a}+1}{q+1}, q^{b}-1\bigg )$  divides $(a, q+1)$;
\item[ vi)] $(q^n+1,q-1)=(2,q-1).$ If $n$ is even then
$(q^n+1,q+1)=(2,q-1).$
\end{itemize}
\end{lemma}

The notion of irreducibility of action on the natural module $V$ is
connected very strictly to that of primitive prime divisor,
which we recall here, for convenience of the reader.\\
Let $t\geq 2$ be a natural number. A prime $q_t$ is said to be a
{\em primitive prime divisor} of $q^t-1$ if $q_t$ divides $q^t-1$
and $q_t$ does not divide $q^i-1$ for any $1\leq i<t.$ It was proved
by Zsigmondy in \cite{zs} that if $t\geq 3$ and the pair $(q,t)$ is
not $(2,6),$ then $q^t-1$ has a primitive prime divisor. If $t=2$ a
primitive prime divisor exists if and only if $q\neq 2^i-1.$\\
Clearly if $q_t$ is a primitive prime divisor of $q^t-1$, then $q$
has order $t$ modulo $q_t$ and thus $t$ divides $q_t-1.$ In
particular $q_t \ge t+1.$
   Let $P_t(q)$ denote the set of primitive prime divisors $q_t$.
\begin{remark}\label{primitivi} Let $P_t(q)$ be the set of primitive prime
divisors  of $q^t-1.$ Then:
\begin{itemize}\item[i)] for any $k\in
\N,\ P_{kt}(q)\subseteq P_t(q^k).$ In particular if $q=q_0^2,$ then
$P_{2t}(q_0)\subseteq P_t(q);$
\item[ii)] if $t\neq s$, then
$P_t(q)\cap P_s(q)=\varnothing;$ \item[iii)] $q_a\in P_t(q)$ divides
$q^b-1$ if and only if $a$ divides $b;$ if $q_a$ divides $q^b+1,$
then a divides $2b.$
\end{itemize}
\end{remark}

 The natural
embeddings of classical (or general classical) groups $G$ of
dimension $d'$ into the corresponding classical groups of dimension
$d>d',$ as well as the embeddings $O_{2n}^{\pm}(q)< O_{2n+1}(q)$ and
$O_{2(n-1)}^{\pm}(q)<O_{2n}(q)^{\mp},$ give rise to some interesting
elements $\sigma$ which we call $low$-$Singer\ cycles.$ Their action
is always reducible and it decomposes $V$ into the direct sum of a
trivial submodule and an irreducible submodule $T.$ Observe that the
natural component for a low-Singer cycle is a suitable subgroup in
$\mathcal{C}_1$  of type orthogonal sum. We call
$t=dim_{_{\,\f{}}}T$ the {\it rank} of $\sigma.$ If $t$ is the
highest dimension of an irreducible $\f{}G$-submodule of $V,$ we
call $\sigma$ a {\em low-Singer cycle of maximal
 rank}.  In Theorem 1.1 of \cite {msw}, are determined the maximal subgroups
   containing a low-Singer cycle of rank $n-1$ in $SL_n(q)$, in $SU_n(q)$
   with $n$ even and in $\Omega_{2n+1}(q).$

 The well known diagonal embeddings of $GL_n(q)$ into
 $Sp_{2n}(q),\ GU_{2n}(q),\ O_{2n}^{+}(q)$ brings into those groups, elements of
 order $q^n-1$ and action $n\oplus n.$ We refer to them as
 \em{linear Singer cycles} of \em{dimension} $2n$.\\

We also need to introduce the fundamental facts about
the theory of the $ppd(d,q;e)$-$elements$ developed by R.
Guralnick, T. Penttila, C. E. Praeger and J. Saxl in \cite{gpps}.
 Through this paper, that theory will be the main tool in finding the maximal
  subgroups containing elements with order divisible by certain "large"
  primes. \\
An element of $GL_d(q)$ is called a $ppd(d,q;e)$-$element$ if its
order is divisible by some $q_e\in P_e(q)$ with $d/2 < e \le d.$ A
subgroup $M$ of $GL_d(q)$ containing a $ppd(d,q;e)$-element is said
to be a $ppd(d,q;e)$-$group.$ In \cite{gpps}
the complete list of the $ppd(d,q;e)$-groups is described.\\
Clearly, if $M\in \bigcup_{i=1}^{8}\mathcal{C}_i \bigcup
\mathcal{S}$ is a $ppd(d,q;e)$-group, then it belongs exactly to an
Example of the list: in particular if $M\in \mathcal{S},$ then $M$
is described in Examples 2.6-2.9.
\begin{thm}\label{main} \cite[Main Theorem]{gpps}
Let $q$ be a prime power and $d$ an integer, $d \ge2$. Then $M\le
GL_d(q)$ is a $ppd(d,q;e)$-group if and only if $M$ is one of the
groups in Examples 2.1-2.9. Moreover
\begin{itemize} \item[i)] $M\not \in \mathcal{C}_4\cup
\mathcal{C}_7;$ \item [ii)] if $e \le d-4,$ then $M$ is not one of
the groups described in the Examples 2.5, 2.6 b), 2.6 c), 2.7,
2.8, 2.9;\item[iii)] if $e=d-3$ and $M$ is one of the groups
described in the Examples 2.5, 2.6 b), 2.6 c), then $d$ is odd.
\end{itemize}
\end{thm}

We shall use the the theory of the $ppd(d,q;e)$-$elements$ for some
special elements.\\
Let $n\geq 5.$  If
  $n \ge 8,$ then by the Bertrand
Postulate, there exists a prime $t$ such that $n/2< t \leq n-3.$ If
$n=5,\ 6,\ 7$ we consider respectively $t=3,\ 4,\ 5,$ getting $\,n/2
< t \le n-2.$ Moreover if $n\neq 6,$ then $t$ is an odd prime with
$(n,t)=1$ and if $n \ge 7,$ then $t \ge 5$ . We call $t$ a {\it
Bertrand number} for $n$  ({\it a Bertrand prime} if $n\neq 6$).
Note that if $t$ is a Bertrand prime for $n$, then
 $\bigg (q^{t}+1, q^{n-t}+(-1)^{n} \bigg )= q+1.$\\
Given a Bertrand number we can define as in Table \ref{2}, an
element $z$ in the classical groups, called a {\it Bertrand
element}. Then $z$ is a $ppd(d,q;e)$-$elements$, as described in
Table \ref{2}.

 \begin{table}[ht]
\caption{Orders of the Bertrand elements $z$} \label{2}
\begin{tabular}{|c|c|c|c|}
\hline  $G$ & $order\ of \ z$ & $d$ & $ e$ \\
 \hline \hline
& & & \\
 $SU_n(q_0^2)$, $n \ne 6$, & $\frac{( q_0^t +1)(
q_0^{n-t} +(-1)^n)}{ q_0+1}$ &  $n$ & $t$\\ & & & \\ \hline
 & & & \\
$Sp _{2n}(q)$  & $\frac{(q^t + 1)(q^{n-t} + 1)}{(q^t +
1,q^{n-t} + 1)}$ & $2n$ & $2t$ \\ & & & \\ \hline & & & \\
$\Omega_{2n+1}( q)$ & $\frac{(q^t + 1)(q^{n-t}+  1)}{(q^t +
1,q^{n-t} + 1)}$ & $2n+1$ & $2t$ \\& & & \\

\hline
\end{tabular}\end{table}
Let $G=Sp_{2n}(q),\ \Omega_{2n+1}(q).$ Then $q_{2t}$ divides $|z|,$
for any $q_{2t}\in P_{2t}(q),$ hence if $M$ is a maximal subgroup of
$G$ containing $z$, as described in Table \ref{2}, then $M\in
\bigcup_{i=1}^{8}\mathcal{C}_i \bigcup \mathcal{S}$ has the
$ppd(d,q;e)$-property and  we can use the description given in
Theorem \ref{main} to determine the maximal
groups containing a Bertrand element.\\
Moreover, since $e \le d-4$ we reduce to the Examples 2.1, 2.2, 2.3,
2.4 and 2.6 a).\\

We will often be concerned with a particular class of
$ppd(d,q;e)$-elements: we say that an element of $GL_d(q)$ is a
$strong$ $ppd(d,q;e)$-$element$ if its order is divisible by any
$q_e\in P_e(q)$ with $d/2 < e \le d.$ Those elements cannot involve
the class $\mathcal{C}_5,$ in the sense of the following:\\

\begin{remark}\label{no-c5}
Let $G_d(q)$ be a classical group and let $q=\tilde q^r,$ for some
prime $r$. If $M<\cdot\ G_d(q)$ with $M\in\mathcal{C}_5$ is of
type $G_d(\tilde q),$ then $M$ contains no strong
$ppd(d,q;e)$-element.
\end{remark}
\begin{proof} Let  $y$ be a strong $ppd(d,q;e)$-element and
$y\in M <\cdot\ G_d(q)$ with $M\in\mathcal{C}_5$ of type
$G_d(\tilde q)$. Then $d/2 < e \le d$ and $\, q_e$ divides
$|y|\,$ for any $q_e \in P_e(q).$ By Remark \ref{primitivi} we
have $P_{re}(q)\subseteq P_{e}(\tilde q^r)=P_e(q)$ and by the
structure of $M$ given in \cite{kl}, we get $(\tilde q)_{re}$ divides
$|G_d(\tilde q)|$ for any $(\tilde q)_{re}\in P_{re}(\tilde q),$
which gives $re\leq d$ against $re\geq 2e>d.$
\end{proof}

Since we usually work with semisimple elements, we shall use some
facts about the maximal tori of the classical groups,
 in particular their orders and their action on
the natural module $V$, which can be easily found in \cite{GLS3}.

We close this section by observing that a $2$-covering for a
classical group exists if and only if it exists for the
corresponding projective group.
\begin{remark}\label{centre}
i)Let $G$ be a perfect group. If $M<\cdot\  G,$ then $M\geq Z(G).$\\
ii) Let $G$ be a classical group. Then $G$ is 2-coverable if and
only if $G/Z(G)$ is 2-coverable.
 \end{remark}
    \begin{proof} i) Assume $M<\cdot\  G,$ with $G$ perfect and
    $Z=Z(G)\not\leq M.$ Then we have $G=MZ$ and $M\lhd G$ with
    $G/M \cong Z/Z\cap M$  abelian. Hence $M\geq G'=G,$ a
    contradiction.\\
    ii) Let $G$ be a classical group and $\{H,\ K\}$ be a
    $2$-covering of $G$ with $H,\ K<\cdot\ G.$ Then $G$ is perfect
    and i) applies, yielding $H,\ K \geq Z(G).$ Thus, using the bar
    notation to take quotient modulo $Z(G),$ we observe that
        $\overline H,\ \overline K\neq \overline G$ are components of a
    $2$-covering of $\overline G.$ The converse is clear.
    \end{proof}
Finally, when considering the simple groups $S$ in ATLAS
\cite{atlas}, we often use the fact that if $M<\cdot\, S$ and $x\in
S,$ then $x\in [M],$ if and only if $\chi_M(x)\neq 0.$

%%%%%%%%%%%%%%%%%%%%%%%%%%%

\section{Symplectic  groups}
Let $G=Sp_{2n}(q)\ $ with $n\geq 5$ and let $s\in G$ be a Singer
cycle. Then  $|s|=q^{n}+1 $ and the maximal subgroups of $G$
containing $s$  are known.
\begin{lemma}\cite[Theorem 1.1]{msw}\label {msw-sp}
Let $G=Sp_{2n}(q)$, $n \ge 5$ and $M$  a maximal subgroup of $G$
containing a Singer cycle. Then, up to conjugacy, one of the
following holds:
\begin{itemize}
\item[i)] $M= Sp_{2n/k}(q^k).k$, with $k|n$ a prime; \item[ii)]
$q$ is  even and $M= O_{2n}^-(q);$ \item[iii)] $nq$ is odd and $M =
GU_n(q^2).2.$ % \item[iv)] $(M,G)=(2^5:A_5, Sp_4(3));$ \item[v)]
%$(M,G)=(PSL_2(17), Sp_8(2)).$
\end{itemize}
\end{lemma}

In order to find the second component for a 2-covering, we consider
a Bertrand element $z\in G$ of order $\frac{(q^{t} +1)(q^{n-t}
+1)}{(q^{t} +1, q^{n-t} +1)}.$ Recall that when $n=6,$ by
definition, $t=4.$ It is clear that, for $(n,q)\neq (5,2),\ z$ is a
strong $ppd(\,2n,\,q;2t)$-element. The group $Sp_{10}(2)$ will be
examined separately.\\ Observe that for $(n-t,q)\neq (3,2)$ any
$q_{2(n-t)}\in P_{2(n-t)}(q)$ divides $|z|;$ when $(n-t,q)=(3,2)$,
we will say that $(n,t,q)$ belongs to the $critical\ case.$
\begin{lemma}\label{bertrand-sp} Let
$G=Sp_{2n}(q),$ with $n \ge 5$  and $(n,q)\neq (5,2).$ If $M$ is a
maximal subgroup of $G$ containing a Bertrand
 element then, up to conjugacy, one of the following holds:
\begin{itemize}
\item[i)] $M = Sp_{2t}(q) \bot Sp_{2n-2t}(q)$;

\item[ii)] $n$ is  even, $q$ is odd  and  $M = GU_{n}(q^{2}).2;$

\item[iii)] $q$ is even and $M \cong O^{+}_{2n}(q).$
\end{itemize}
\end{lemma}
\begin{proof}
Let $M<\cdot\ G=Sp_{2n}(q),$ containing a Bertrand element $z.$
Since $2t\leq 2n-4,$ by Theorem \ref{main} and Remark \ref{no-c5},
$M$ belongs to one of the classes $\mathcal C_i$, $i=1,2,3,8$ or to
$\mathcal S$ as described in Example 2.6 a) of \cite{gpps}.

\vskip 3pt

\noindent $\mathcal{C}_1$. \ \ \ Suppose first that $(n,t,q)$ does
not belong to the critical cases.  Then $q_{2t} \cdot q_{2(n-t)}$
 divides the order of an element of $M$. Suppose that $M$ is
of type $P_m$ with $1 \le m \le n$. Then, by Proposition 4.1.19 of
\cite{kl}, we have $P_m \cong q^a: GL_{m}(q) \times Sp_{2n
-2m}(q)$ for some $a \in \N$. This implies  $n-m \ge t$ and  $m
\ge 2n -2t$, which  gives $n \le t$ against the definition of $t.$
Now suppose that $(n,t,q)$ belongs to the critical cases. Then
$n\neq 6$ and we have again $n-m \ge t$, which gives $m \le 3$.
But, since $3\mid (2^t+1),$ there is no element of order
$|z|=3(2^t+1)$ in $P_{1}$, $P_{2}$ or $P_{3}$.

Suppose now that $M$ is of type $Sp_{2m}(q) \bot Sp_{2n-2m}(q)$,
with $ 1 \le  m < n$ and  $(n-t,q) \ne (3,2)$. Then  $n-m \ge t$ and
$n -t\le m$, which gives $m=n-t$. If $t=n-3$ and $q=2$, we also have
$m \le 3.$ On the other hand $|z|=3(2^t+1)$ cannot be the order of
an element in $Sp_2(2) \bot S_{2n-2}( 2)$ or in $ Sp_4(2) \bot
Sp_{2n-4}( 2)$ and we are left only with $Sp_6(2) \bot Sp_{2t}( 2)$.

\vskip 3pt \noindent $\mathcal{C}_2$. \ \ \ By definition of the
$\mathcal{C}_2$ class of the symplectic
   group, $M\cong Sp_m(q) \wr Sym(2n/m)$, preserves a direct sum decomposition
   $V=V_1\oplus\cdots\oplus V_k$ where each subspace $V_i$ has even dimension $m.$
    On the other hand, $M$ is described in
   Example 2.3 of \cite{gpps}: in particular, by Lemma 4.1 in
   \cite{gpps}, $m=1.$ Hence no case arises.

 \vskip 3pt \noindent
$\mathcal{C}_3$. \ \ \ If $M \cong Sp_{2n/r}(q^{r}).r$, where
$r\mid n$ is a prime, then $q_{2t}$ cannot divide $|M|.$ Namely
$$\pi(|M|)=\pi\bigg (p\, r \, \prod_{i=0}^{n/r}(q^{2ri} - 1) \cdot
\bigg )$$  and if  $q_{2t}$ would divide $q^{ri} \pm 1$,  we
should have $t \le i\le n/r$ against $t >n/2;$ moreover $q_{2t} >
n \ge r$ implies $q_{2t} \ne r$.

If $M \cong GU_{n}(q^{2}).2$ , then $z$ belongs to $M$ if and only
if $n$ is even.

\vskip 3pt \noindent $\mathcal{C}_8$. \ \ \ Here $M=
O_{2n}^{\pm}(q)$, with $q$ even  and
only $O_{2n}^{+}(q)$ contains an element of order divisible by
$q_{2t}\,q_{2(n-t)}.$
%%%%%%%%%%%%%%%%%%%%%%%%%%%%%%%%

\vskip 4pt \noindent $ \mathcal{S}.$ \ \ \ In the Example 2.6 a)
of \cite{gpps} we found $M \le Sym(m) \times Z(G)$ with $m\in\{
 2n+1,\ 2n+2\}$ and $q_{2t}=2t+1.$ Assume $n \ge 7,$ then $t \ge 5$ and first suppose
that $(n-t,q)=(3,2)$. Then there exists in $M$ a cyclic subgroup of
order $2_{2t}\,9 $, against the fact that $2t +1 +9= 2n +4 > m$.
Next let $(n-t, q) \ne (3,2)$: then there exists a primitive prime
divisor $q_{2(n -t)}$ of $q^{2(n-t)} -1$ and $q_{2(n -t)} \ge 2n -2t
+1.$ We observe that if $(t,q) \ne (5,2)$, then $(q^t +1)/(q+1) > 2t
+1=q_{2t}.$ Since, by Lemma \ref{aritme},  $(q^t +1)/(q+1)$ is odd,
there exists an odd prime $r$ such that $r \cdot q_{2t}$ divides
$(q^t +1)/(q+1).$ Moreover, by Lemma \ref{aritme}, ${(q^{n-t} +1,
\frac{q^t +1}{q+1})}$ divides $(t, q+1)$ which together with
$n-t\geq 2,$ gives $r\neq q_{2(n-t)}.$ Thus if $r \ne q_{2t}$, we
have that $r \cdot q_{2t} \cdot q_{2(n -t)}$ divides $|z|$ which
requires an element  of order the product of these three primes in
 $Sym(m),$ and therefore  $m\geq r + q_{2t} + q_{2(n -t)}
  \ge 3 + 2n -2t +1 + 2t+1 = 2n+5,$ against $m\leq 2n+2.$ If $r=q_{2t}$ we have an element of order
$q_{2(n-t)}\,q_{2t}^2$ in $Sym(m),$ which implies the impossible
relation $m \geq q_{2(n-t)}+q_{2t}^2> 2n+2.$

In the case $(t,q)=(5,2),$  we have $n\in \{7,\ 8,\ 9\}$, but it
is easily checked that the corresponding $Sym(m)$ do not contain
elements of order $|z|$.

Now assume  $n=5$ and $q\ne 2$. Then $t=3,$ $q_{2t}=7$, and $m=11$
or $12$. Let $\sigma$ be an element of $M$ of order
 7: then $(q^2 +1)/(2,q-1)$ divides $|C_M(\sigma)|,$ hence
$q_4$
 divides the order of $C_{Sym(m)\times Z(G)}(\sigma)\cong Sym(m-7)
 \times C_7 \times Z(G).$ Since $7\neq q_4\geq 5 ,$ we get
  $m=12,\ q_4=5.$ Thus we would have an element $z$ of order divisible
   by  $35$ in $Sym(12),$ which implies the impossible relation $|z|=\frac{(q^3+1)(q^2+1)}
 {(2,q-1)}=35.$  \\
 Finally observe that if $n=6,$ then $t=4$ and no case arises
  since
  $2t+1=9$ is not a prime.
   \end{proof}

\begin{remark}\label{$Sp_10(2)$} If $\{H,\ K\}$ is a 2-covering of
$Sp_{10}(2),$ then $$H= O^-_{10}(2)\ \ \ and \ \ \ K=
O^{+}_{10}(2). $$

\begin{proof} Let $\{H,\ K\}$ be a 2-covering of $G=Sp_{10}(2),$
with $H$ containing a Singer cycle of order 33. By Lemma
\ref{msw-sp}, we get $$H\in \{Sp_2(2^5).5,\ O^{-}_{10}(2)\}.$$ Let
$y\in G$  of order $17\cdot 3$ and action of type $8\oplus 2.$ Then
$y$ is a $ppd(\,d,q;e\,)$-element for $d=10,\ q=2,\ e=8.$ Observe
that $e=d-2,\ q_8=2e+1=17.$ The maximal subgroups $M$ of $G$
containing $y$ belong to one of the classes $\mathcal C_i$,
$i=1,\,2,\,3,\,6,\,8$ or to $\mathcal S$ and are described closely
in the Examples of \cite{gpps}. Since $y$ has no eigenvalue, the
only $M\in \mathcal C_1,$ is the natural $M=Sp_2(2)\bot Sp_8(2);$ no
case arises in classes $\mathcal C_2,\ \mathcal C_6,$ since the
corresponding Examples 2.3 and 2.5 in \cite{gpps} require $q_e=e+1;$
we get no case also in $\mathcal C_3$ by arithmetical reasons and
finally in $\mathcal C_8,$ due to the action of $y,$ we get only
$M=O^{+}_{10}(2).$ On the other hand the examination of the Examples
2.6-2.9 in \cite{gpps} for the class $\mathcal S,$ easily show that
there is no possibility for $M$ in $\mathcal S.$ Thus we reach
$$K\in\{Sp_2(2)\bot Sp_8(2),\ O^{+}_{10}(2)\}.$$ Let $z\in G$ be a
Bertrand element of order 45 and observe that its action is of type
$6\oplus 4.$ The only subgroup among our candidates $H$ and $K$
which contain $z$ is $O^{+}_{10}(2).$ Moreover $O^{+}_{10}(2)$ and
$Sp_2(2^5).5$ cannot constitute the components of a 2-covering,
since $G$ contains elements of order 35 which none of them contains.
Thus we are left with the only possibility $H= O^-_{10}(2)$ and $ K=
O^{+}_{10}(2). $
\end{proof}
\end{remark}
In \cite{Dye},R. H. Dye showed that in even characteristic the
symplectic group admit always a 2-covering:\\
\begin{thm}\cite{Dye}\label{Dye}
 The group $Sp_{2n}(2^f)$ is 2-coverable by
$$H= O^-_{2n}(q)\quad\hbox{and }\quad K= O^{+}_{2n}(q).$$
\end{thm}
\begin{prop}\label{simplettici}
Let  $G=Sp_{2n}( q))$, $n\ge 5$.
\begin{itemize}
\item[i)]  If $q$ is odd, then $G$ is not 2-coverable; \item[ii)]
if $q$ is even, then the only  2-covering  $\{H,K \}$ of $G$ is
given by
$$H= O^-_{2n}(q)\ \ \ and \ \ \ K= O^{+}_{2n}(q). \ \ \
    $$
\end{itemize}
\end{prop}
\begin{proof} By Remark \ref{$Sp_10(2)$} and Theorem \ref{Dye}, the result is
 true if $(n,q)= (5,2).$
Let  $n \ge 5$ with $(n,q)\neq (5,2)$ and  $\{H,\ K \}$ be a
2-covering of $G=Sp_{2n}( q)$.

We can assume that $H$ contains a Singer cycle and therefore it
 is described in Lemma \ref{msw-sp}. On the other hand, the maximal
  subgroups of $G$
containing a Bertrand element are described in Lemma
\ref{bertrand-sp} and, since there is no overlap between these two
lists, we can assume $K$ as described there. Thus we have two
choices for $H$ and two choices for $K$ both if $q$ is odd and if
$q$ is even. Let $y\in G$  of order $\frac{(q^{n-1}
+1)(q+1)}{(q^{n-1} +1,q+1)}$  and action of type $2(n-1)\oplus 2.$\\
Suppose first that $q$ is odd. Then $|y|$ does not divide $|H|$.
Moreover  $|y|$ does not divide $|Sp_{2t}(q) \bot Sp_{2n-2t}(q)|.$
Hence if $n$ is odd we have finished. If $n$ is even, the only
possibility to contain $y$ is given by the choice $K=
GU_{n}(q^{2}).2.$ But it is easily observed that an element of
order $q^{n-1} -1$ is not contained neither in $H=Sp(n/k,q^k).k$
nor in $K=GU_{n}(q^{2})$.
  Thus we get no 2-covering when $q$ is odd.\\ Now
suppose that $q$ is even. Then $K =O_{2n}^+(q),$ since no other
 candidate component can contain elements with the
order and action of $y.$ Thus we have two possible coverings given
by $$\delta_1=\{H=O_{2n}^-(q),\ \ K =O_{2n}^+(q)\}$$ and
$$\delta_2=\{H=Sp_{2n/k}(q^k).k,\ \  K =O_{2n}^+(q)\},$$ where $k|n$ is a
prime. Finally we see that $\delta_2$ is never a 2-covering.

We first suppose that $n$ is odd. Then  $(q^{2} +1,q^{n-2}-1)=1$,
since $n-2$ is odd and $q$ is even. Then there exists an element
$u\in G$ of order $(q^{2} +1)(q^{n-2}-1)$ and $q_{n-2}$ divides
$|u|$. Since $n-2$ divides $q_{n-2}-1$ we get $q_{n-2}> n \ge k,$
hence $q_{n-2}$ does not divide $|Sp_{2n/k}(q^k).k|.$ Moreover $u$ cannot belong to $O_{2n}^+(q).$

Now suppose that $n$ is even. Then $(q^{n-1} -1,q +1)=1,$ since
$q$ is even and $n-1$ is odd. Then there exists $u\in G$ of order
$(q^{n-1}-1)(q +1)$ and $q_{n-1}$ divides $|u|$. But $q_{n-1}>n$
does not divide $|Sp_{2n/k}(q^k).k|$  and there is no element of order $|u|$ in
$O_{2n}^+(q).$
\end{proof}

%%%%%%%%%%%%%%%%%%
\section{Unitary groups}

Let $G=SU_{n}(q)\ $ with $n\geq 5,\ n\neq 6$
  and $q=q_0^2.$
  Let
$s\in G$ be a Singer cycle of order  $(q_0^{n}+1)/(q_0+1)$ if $n$ is
odd and a low-Singer cycle of maximal rank of order $q_0^{n-1}+1$ if
$n$ is even.

\begin{lemma}\cite[Theorem 1.1]{msw}\label {malleu}
Let $G=SU_n(q)$, $n \ge 5,\ n\neq 6.$ If $M$ is a maximal subgroup
of $G$ containing $s$, then one of the following holds:
\begin{itemize}
     \item[i)] $n$ is odd, with $(n,q_0)\neq (5,2)$ and  $$M
= N_G(GU_{n/k} (q^k))\cong
SU_{n/k}(q^k).\,\frac{q_0^k+1}{q_0+1}\,.\,k,$$ with  $\ k|n$
prime; \item[ii)] $n$ is  even and  $M = GU_{n-1}(q);$
 \item[iii)] $(M,G)= (PSL_2(11),\,SU_5(4)).$
 \end{itemize}
 \end{lemma}

\begin{lemma}\label{$SU_5(4)$} The group $SU_5(4)$ is not $2$-coverable.
\end{lemma}
\begin{proof} We use \cite{atlas} for the description of the conjugacy
   classes of elements in $G$ and for the maximal subgroups in $G.$
   Let $\{H,K\}$ be a 2-covering of $G$ and observe that $G$ contains
one conjugacy class of elements of order 8 and three conjugacy class
of elements of order 9. By Lemma \ref{malleu}, we can assume
$H=PSL_2(11)$ and since this group does not contain elements of
order 8 or 9, $K$ must contain, up to conjugacy, all of them. But
since the character $\chi_{_{P_1}}$ vanishes on the class 9C, we
have $9C\notin[P_1].$ All the other maximal subgroups in $G$ do not
contain elements of order 8.
\end{proof}

From now on we will always assume that $q_0\neq 2$ when $n=5.$ \\
Let $t$ be a Bertrand prime for $n$ and consider a Bertrand element
$z\in SU_n(q)$ as in Table \ref{2}. Then $|z|=\frac{(q_0^t
+1)(q_0^{n-t} +(-1)^n)}{q_0+1}$ and the action on $V$ is of type
$t\oplus\frac{(n-t)}{2}\oplus \frac{(n-t)}{2}$ if $n$ is odd and of
type $t\oplus (n-t)$ if $n$ is even.

For the unitary groups we have either   $e \le d-3$,  or $(d,e)\in
\{(5,3),\ (7,5)\}.$ Hence Theorem \ref{main} applies and we search
$M$ among the groups in the Examples 2.1-2.9, for $n\geq 5.$
 Recall that $t\geq 5$ for $n\neq 5$ and note
that, for any $(q_0)_{2t}\in P_{2t}(q_0),$ we have $(q_0)_{2t}
\mid |z|.$ Moreover if $n$ is even, for any $(q_0)_{2(n-t)}\in
P_{2(n-t)}(q_0)$ we have $(q_0)_{2(n-t)} \mid |z|,$ while if $n$
is odd for any
$(q_0)_{(n-t)}\in P_{(n-t)}(q_0)$ we have $(q_0)_{(n-t)}\mid |z|.$\\
Note also that, due to the exceptions to Zsigmondy theorem, for
$n$ odd with $n-t=6,\ q_0=2,$ the set $P_{(n-t)}(q_0)$ is empty
and also when $n$ is even and $n-t=3,\ q_0=2,$ the set
$P_{2(n-t)}(q_0)$ is empty. We refer to these two situations
saying that $(n,t,q_0)$ belongs to the {\em the critical cases}.\\
%Observe that $P_{(n-t)}(q_0)$ is empty even if $n=5$ and
%$q_0+1=2^i,$ but we will never front this eventuality.
%SITUAZIONE MAI SPECIFICATA ALTROVE, DUNQUE ELIMINABILE. CONTROLLARE!!
\begin{lemma}\label{uni-z}
Let $G=SU_n(q)$ $n\ge 5$, $n \ne 6$. If $M$ is a maximal subgroup
of $G$ containing a Bertrand element, then one of the following
holds:
\begin{itemize}
\item [i)] $M= SU_t(q)\bot GU_{n-t}(q);$ \item [ii)] $n$ is odd
and $M =P_{(n-t)/2};$ \item[iii)] $(n,q_0) = (7,2)$, and $M =
GU_{n-1}(q)$.
\end{itemize}
\end{lemma}

\begin{proof}
Let $M$ be a maximal subgroup of $SU_n(q)$ containing the Bertrand
element $z$. Then, by Remark \ref{primitivi}, $M\in
\bigcup_{i=1}^{7}\mathcal{C}_i \bigcup \mathcal{S}$ has the
$ppd(n,q;t)$-property, with $n \ge 8,$ $n/2 < t \le n-3$, or
$(n,t)\in \{(5,3),\ (7,5)\}.$ Thus Theorem \ref{main} applies and
we search $M$ among the groups in the Examples 2.1-2.9, for $n\geq
5.$ But obviously $q_t \ne t+1$ and $(t,n)=1.$ Moreover, looking
at the Tables 2-8 of \cite{gpps}, it is easily checked that $q_t
\neq 2 t+1$ since $t\leq n-2.$ These facts rules out the groups of
the examples 2.3, 2.4, 2.5, 2.6, 2.7, 2.8, 2.9, because they are
all given under at least one of the conditions: $t> n-3$ and
$n\neq 5,\ 7,$ or $t=n-3$ even, or $q_t = t+1,$ or $q_t=2t+1.$

Hence we reduce to $M\in \mathcal{C}_1$ or $M \in \mathcal{C}_5.$

Let $M\in \mathcal{C}_1.$ If $M \cong SU_m(q)
   \bot \ GU_{n-m}(q))$, with $m<n/2,$
the condition $z \in M$, implies that $(q_0)_{2t}$ divides the
order of a maximal torus either of $SU_m(q)$ or of $GU_{n-m}(q),$
hence $(q_0)_{2t}$ divides
$$\alpha=\prod_{i=1}^k(q_0^{s_i} -(-1)^{s_{i}}) \prod_{i=1}^l(q_0^{r_i}
-(-1)^{r_{i}}),$$ with
       $\sum_{i=1}^{k}s_{i}=m\ $ and $\ \sum_{i=1}^{l}r_{i}=n-m.$
    Since $m<n/2<t,$ we must have $n-m\geq t.$ For $n=5,\  7$ this
produces $m=1$ or $m=2.$ On the other hand $GU_4(q)$ does not
contain an element of order $(q_0^3+1)(q_0-1),$ if $q_0 \neq 2$
and $U(6,q)$  contains an element of order $(q_0^5+1)(q_0-1),$
only if $q_0 =2$. Then, with the exception of $SU_7(4),$ we get
 again
$m=2=n-t.$

For $n>7,$ we  suppose that $(n,t,q_0)$ do not belong to the
critical cases. Hence if
    $m<n-t,$ the condition
$(q_0)_{(n-t)}|\alpha$ for $n$ odd or $(q_0)_{2(n-t)}|\alpha$ for
$n$ even couldn't be fulfilled. Then we are left with  $m=n-t.$

We now suppose that  $(n,t,q_0)$ belongs to the critical cases.
Let first $n$ be odd, $n-t=6$, $q_0=2$. Then  $m \le 6$ and $|z|=
(2^t +1) \cdot 21$. But $7 \cdot 2_{2t}$ divide the order of an
element in $SU_m(4)\bot
GU_{n-m}(4)$ if and only if $m=6=n-t$.

If $n$ is even, $n-t=3$, $q_0=2$, then  $m \le 3$ and $|z|= (2^t
+1)\cdot 3$. Observe that if  $m=1,2$ there is no element of order
$|z|$ in $GU_{n-m}(q)\bot SU_m(q).$
%, since the maximal order of a semisimple element in this group is $(2^t +1)$
%NON e' corretto.
Thus again $m=3=n-t.$

    If $$M=P_m \cong q_0^{n(2n-3m)}\ :\ \frac{1}{q_0+1}(GL_m(q)\times GU_{n-2m}(q)),$$
     we have that $$|M|_{q_0'}=
(q_0^2-1) \prod_{i=2}^m(q^{i} -1)\ \prod_{i=2}^{n-2m}( q_0^{i}
-(-1)^i)$$ and $m\leq n/2<t.$ By Lemma \ref{primitivi},
$(q_0)_{2t}$ is also a primitive prime divisor $q_t$ for $q^t-1$
and divides $|z|.$ Since $q_t$ does not divide
$\prod_{i=2}^m(q^{i} -1)$ we need $(q_0)_{2t}$ divides $(q_0^i-(-1)^i)$ for
some $i\leq n-2m$ which gives $m\leq (n-t)/2.$

Now suppose that $M$ is of type $P_m$ and  that $(n,t,q_0)$ do not
belong to the critical cases.
    Let $m\le (n-t)/2$ and $n$ even;
    then $(q_0)_{2t}\,(q_0)_{2(n-t)}$ divides $|z|$ and,
    there exists an element with order divisible by
$(q_0)_{2t}\,(q_0)_{2(n-t)}$ in $GL_m(q)\times SU_{n-2m}(q).$ The
possible order of such an element is a divisor of
$$\beta=\prod_{i=1}^k(q^{r_i} -1) \prod_{j=1}^l(q_0^{s_j}
-(-1)^{s_{j}})   ,$$  with
       $\sum_{i=1}^{k}r_{i}=m\ $ and $\ \sum_{j=1}^{l}s_{j}=n-2m\ (\geq
t).$

Since $m<t$ the prime $(q_0)_{2t}=q_t$  does not divide
$\prod_{i=1}^k(q^{r_i} -1)$ and we must have, say, $t\mid s_1$ and
$\ \sum_{j=2}^{l}s_{j}\leq n-2m-t<n-t.$ Then $(q_0)_{2(n-t)}$ does
not divide $\prod_{j=1}^l(q_0^{s_j} -(-1)^{s_{j}}) $ and we need
$(q_0)_{2(n-t)}=q_{n-t}$ divides $q^{r_i}-1,$ hence the
contradiction $n-t\leq r_i\leq m\leq (n-t)/2.$\\ Let $m \le
(n-t)/2$ and $n$ odd. If $n=5,\ 7$ we have finished. If $n\geq 9,$
then there exist primitive prime divisors $(q_0)_{2t},\
(q_0)_{(n-t)}$ and we have that $(q_0)_{2t}\,(q_0)_{(n-t)}$
divides $$\beta=\prod_{i=1}^k(q^{r_i} -1) \prod_{j=1}^l(q_0^{s_j}
-(-1)^{s_{j}}),$$ with
       $\sum_{i=1}^{k}r_{i}=m\ $ and $\ \sum_{j=1}^{l}s_{j}=n-2m\ (\geq
       t).$ Hence $t\mid s_1$ and $\
\sum_{j=2}^{l}s_{j}\leq n-2m-t<n-t,$ which implies that
$(q_0)_{(n-t)}=q_{(n-t)/2}$ divides $\prod_{i=1}^k(q^{r_i} -1)$
and then $(n-t)/2\leq m.$ Thus $m=(n-t)/2$ and $M=P_{(n-t)/2}.$

Now suppose that  $(n,t,q_0)$  belongs to the critical cases. If
$n$ is even, $n-t=3$ and $q_0=2,$ then  $m \le (n-t)/2 =3/2$ gives
the only case $m=1.$ On the other hand the action of type $t\oplus
3$ of $z$ on $V$ is not compatible with $z\in P_1.$

If  $n$ is odd, $n-t=6$, $q_0=2$. Then $m \le 3$ and $|z|= (2^t
+1)\cdot 21$. Then $7 \neq 2_{2t},$ since $2_{2t}\geq 2t+1\geq 11$
and $7\cdot 2_{2t}$  divide the order of an element  of $P_m\cong
q_0^{\frac{(n-m)(n+m-1)}{2}}\ :\ GL_m(q)\times SU_{n-2m}(q)$ if
and only if $m=3=(n-t)/2.$
    \\
%%%%%%%%%%%%%%%%%%%%%%%%%%%%%%%%%%%%%%%%%%%%%%%%%%%%%%%%%%%%%%%%%%%%%
Next let $M\in \mathcal{C}_5.$ If $M=N_G(GU_{n}(\tilde{q}^2))$
with $\tilde{q}^r=q_0$ and $r$ an odd prime, then arguing as in
Remark \ref{no-c5}, we conclude that $M$ is not of this type.

If $M$ is of type $SO_n^{\epsilon}(q_0)$, with $q_0$ odd, or of
type $Sp_{n}(q_0)$, with $n$ even, we have $$\pi(|M|)\subseteq
     \pi(q_0 \prod_{i=1}^{[n/2]}(q^{i} -1))$$ and
     $(q_0)_{2t}=q_t$ divides $|M|$  gives $t\leq i\leq n/2,$ a
     contradiction.
\end{proof}

\begin{prop}\label{unitari}
Let  $G= SU_{n}(q)$, $n\ge 5,$ with $n\neq 6.$ Then $G$ is not
2-coverable.
\end{prop}
\begin{proof}
        Let $n
\ge 5,$ with $n\neq 6$ and assume, by contradiction, that $\{H,K
\}$ is a 2-covering of $G$. Let first $n$ be even. Then,
     by Lemma \ref{malleu},  $H \cong  GU_{n-1}(q).$ Moreover, by Lemma \ref{uni-z},
      we have $K=SU_t( q)
       \bot GU_{n-t}(q).$

We consider an element $y\in G$ of order $\frac{q_0^n-1}{q_0+1}$.
Then $(q_0)_n$ divides $|y|.$ If $n/2$ is even, then $(q_0)_{n}$
divides neither $|H|$ nor $|SU_t( q)
       \bot U_{n-t}(q)|.$
        If $n/2$ is odd, then there
exists a  primitive  prime divisor $(q_0)_{n/2}$  and
$(q_0)_{n/2}$ divides $|y|,$ while $(q_0)_{n/2}$ divides neither
$|H|$ nor $|SU_t( q)\bot U_{n-t}(q)|.$

We now suppose that $n$ is odd. Then  $H=
SU_{n/k}(q^k).\,\frac{q_0^k+1}{q_0+1}\,.k,$ with  $\ k|n$ a prime.
Let first $n \ge 9;$ then $K=SU_t( q) \bot U_{n-t}( q)$ or
$K=P_{(n-t)/2}.$ We consider a low-Singer cycle $y\in G$ of order
$q_0^{n-2}+1$. Then $ (q_0)_{2(n-2)}$ divides the order of $y$,
but it
divides neither $|H|$ nor $|K|$.\\
   If $n=5,$ by Lemma \ref{$SU_5(4)$}, we can assume $q_0\neq 2$ and
   pick  $y\in G$ with $|y|=q_0^4-1.$
   Assume $y\in H\cong\frac{q_0^5+1}{q_0+1}\,. \,5;$ then
    $(q_0)_4=5,$  which gives $\frac {q_0^4-1}{5}$ divides
$(\frac{q_0^5+1}{q_0+1},q_0^4-1)$ which in turn divides
  $ 5$ against $q_0\neq 2.$ It is also
    clear that $(q_0)_4$ does not divide $|P_1|,\
    |SU_3(q)\bot GU_2(q)|.$

   If $n=7$ and $q_0\neq 2,$ we
consider  $y\in G$ of order $\frac{q_0^4
    -1)(q_0^3+1)}{q_0+1}$. Then $ (q_0)_4 \cdot  (q_0)_6$
divides $|y|.$ Since $(q_0)_4\neq 7,$ it does not divide
$\frac{q_0^7+1}{q_0+1}\,. \,7;$ on the other hand $(q_0)_4 \cdot
(q_0)_6$ does not divide $|SU_5(q) \bot GU_2( q)|.$

Moreover $ P_1$  contains no element of order divisible by
$(q_0)_4 \cdot  (q_0)_6.$

If $n=7$, $q_0=2$, then $|y|=45$ does not divide the order of
$|H|=43\cdot 7.$ It is also clear that no element of order 45
belongs to $SU_5(4)\bot U_2(4)$ or to $ P_1\cong 2^{21}:C_3\times
SU_5(4)$ or to $GU_6(4)$.
\end{proof}

%%%%%%%%%%%%%%%%%%%%%%
\section{Orthogonal groups in odd dimension}

Let $G=\Omega_{2n+1}(q)$ with $n\geq 5$ and consider the maximal
subgroups of $G$, containing a low-Singer cycle $x$ of maximal rank.
Then $x$ has order $(q^{n}+1)/2,$  action of type $2n \oplus 1$ and
the maximal groups containing $x$ are known. Recall that $q$ is
automatically odd.

\begin{lemma}\cite[Theorem 1.1]{msw}\label {msw-odd} Let
$G=\Omega_{2n+1}(q)$, $n \ge 5$  and  $x\in G$ be an element of
order $(q^{n}+1)/2.$  If $M$ is a maximal subgroup of $G$ containing
$x$ then, up to conjugacy, $M$ is isomorphic to
$\Omega_{2n}^-(q).2.$
\end{lemma}

\begin{lemma}\label{bertrand-odd}
Let $G=\Omega_{2n+1}(q)$, $n \ge 5$ and $M$ a maximal subgroup of
$G$ containing a Bertrand element. Then, up to conjugacy, one of the
following holds:
\begin{itemize}
     \item[i)] $M=\Omega_{2n+1}(q)\cap(O_{2t+1}(q)\, \bot\,
O^{-}_{2(n-t)}(q))$; \item[ii)]
$M=\Omega_{2n+1}(q)\cap(O_{2(n-t)+1}(q) \,\bot\, O^{-}_{2t}(q));$
\item[iii)] $M= \Omega_{2n+1}(q)\cap O_{2n}^+(q).$
\end{itemize}
\end{lemma}
\begin{proof} Let $G=\Omega_{2n+1}(q)$, $n \ge 5$ and $M$ a maximal subgroup of
$G$ containing a Bertrand element $z.$ Observe that, since $q$ is
odd and $2n-2t \ge 4$, then there exist primitive prime divisors
$q_{2t}$ and $q_{2(n-t)}$ and $q_{2t}\cdot q_{2(n-t)}$ divides $|z|$
for any $q_{2t}\in P_{2t}(q)$ and for any $q_{2(n-t)}\in
P_{2(n-t)}(q).$ In particular $z$ is a strong $ppd(d,q;e)$-element,
where $d=2n+1$ and $e=2t.$ By Theorem \ref{main}, Remark \ref{no-c5}
and by Table 3.5.D in \cite{kl}, $M$ belongs to one of the classes
$\mathcal C_i$, $i=1,2,3$ or to $\mathcal{S}$ and is described in
Example 2.6 a) of \cite{gpps}.

\vskip 3pt

\noindent $\mathcal{C}_1$. \ \ \ Suppose that $M$ is of type $P_m$
and that $q_{2t} \cdot q_{2(n-t)}$ divides the order of an element
in $M.$ Then, by Proposition 4.1.20 of \cite{kl}, we have $m \ge 2n
-2t$ and $n-m \ge t$ which gives $n \le t$ against the
definition of $t.$ \\
Suppose now that $M=\Omega_{2n+1}(q)\cap (O_{2k+1}(q) \bot
O^{\epsilon}_{2(n-k)}( q))$  and  refer to Proposition 4.1.6 in
\cite{kl} for its structure. If $k=0$ we must choose $\epsilon=+1$.
If $1 \le k \le n-1,$ then we must have either $2(n-k) \ge 2t$ and
$n-t \le k$, which gives $k=n-t$ or $t \le k$ and $n-t\le n-k$, that
is $k=t$. On the other hand it is clear that both these choices
works if and only if we select the minus sign.

     %%%%%%%%%%%%%%%%%%%%%%%%%%%%%
\vskip 4pt

\noindent $\mathcal{C}_2$.\ \ \ These are the groups in the Example
2.3 of \cite{gpps}. Thus we have $q_{2t}=2t +1 $ and $M \le
GL_1(q)\, \wr\, Sym(2n+1)$. To guarantee that $q_{2t}\cdot
q_{2n-2t}$ divides the order of an element in $Sym(2n+1)$, we need
$2n+1 \ge q_{2t} + q_{2n-2t} \ge 2t +1 +2n -2t +1=2n +2$, a
contradiction.
%%%%%%%%%%%%%%%%%%
\vskip 4pt

\noindent $\mathcal{C}_3$. \ \ \ These groups are described in
Example 2.4.  Since $d\neq e+1$ we consider only the case b) of
Example 2.4. Let $b>1$ a divisor of  $(2n+1, 2t).$ Then $n\neq 6$
and $b=t=(2n+1)/3.$ Thus, by Proposition 4.3.17 in \cite{kl},
$M=\Omega_3( q^t).\,t$ has order $\frac{q^{2t}(q^{2t}-1)\,t}{2}$ and
the condition $q_{2(n-t)}$
 divides $|M|$ implies that $q_{2(n-t)}=t.$ Thus $t$ does not divide $|\Omega_3( q^t)|$
  and the only elements of order $t$ in $M$ are the
 field automorphisms $\alpha.$
However, $q_{2t} \not | \ |C_{\Omega_3( q^t)}(\alpha)| = |\Omega_3(
q)|$ and we get no examples in this class.

%%%%%%%%%%%%%%%%%%%%%%%%%%%%%

\vskip 4pt \noindent $\mathcal{S}$. \ \ \ The maximal subgroups $M$
in Example 2.6 a) satisfy  $ M \le Sym(m) ,$  with $m= 2n +2$, if
$p$ does not divide $m$, or $m=2n+3$ if $p$ divides $m$. Moreover
$q_{2t}= 2t +1$, and $ q_{2t} \le m$. Let first $n \ge 7.$ Then $t
\ge 5$ and $(q^t +1)/(q+1) > 2t +1=q_e.$ Then arguing as in the
symplectic case in Lemma \ref{bertrand-sp}, we get an odd prime
$r\neq q_{2(n-t)}$ such that  $r \cdot q_e \cdot q_{2n -2t}$ divides
$|z|,$ which require the impossible relation $m>2n+3.$\\ If $n=5$,
then $q_6=7$ and $m=12$ or $13$ and again we can use the same
argument as in the symplectic case. If $n=6$ no case arises since
$2t+1=9$ is not a prime.
\end{proof}

\begin{prop}\label{ortogonali dispari}
Let  $G=\Omega_{2n+1}(q)$, $n\ge 5$. Then $G$ is not 2-coverable.
\end{prop}
\begin{proof}
Let $n \ge 5$ and assume, by contradiction, that  $\{H,K \}$ is a
2-covering of $G=\Omega_{2n+1}(q)$ with maximal components. Then, by
Lemma \ref{msw-odd} ,
 $H=(\Omega_{2n}^-(q)).2$ and $K$ is given in Lemma
 \ref{bertrand-odd}. If $$K\in \{\Omega_{2n+1}(q)\cap(O_{2t+1}(q)\, \bot\,
O^{-}_{2(n-t)}(q)), \Omega_{2n+1}(q)\cap(O_{2(n-t)+1}(q) \,\bot\,
O^{-}_{2t}(q))\}$$ we consider an element $y\in G$ of order $q^n
-1:$ then $y$ is not contained neither in $H$ nor in $K.$ If
$K=\Omega_{2n+1}(q)\cap O_{2n}^+(q),$ we observe that neither $H$
nor $K$ contains regular unipotent elements.

\end{proof}
%%%%%%%%%%%%%%%%%%%%%%%%%

\section{Orthogonal groups with Witt defect $0$}

Let  $G=\Omega^{+}_{2n}(q),$ with $n\geq 5.$ First of all, to
control the action of some crucial elements in $G$  we need the
following Lemma.
\begin{lemma}\label{action}The generator of $\Omega_2^-(q)\cong C_{\frac{q+1}{(2,q-1)}}$
operates irreducibly on the natural module $V$ if and only if $q\neq
3.$
\end{lemma}
\begin{proof}
 Let $\Omega_2^-(q)=<x>.$ If $q$ is even, then $x$ is the Singer
cycle of $SL_2(q)$ and operates irreducibly. If $q$ is odd, we
observe that $x=\pi_{u^{2(q-1)}},$ where $<u>=\f{2}^*$ and,  by
Lemma 2.4 in \cite{bl}, the minimal polynomial $m(x)$  is
irreducible of degree $r\mid 2,$ minimal with respect to
$\frac{q^2-1}{q^r-1}\mid 2(q-1).$ If $r=1,$ we obtain $q+1\mid
2(q-1),$  which implies that 2 is the only prime divisor of $q+1,$
that is $q=2^i-1$ and hence $q=3.$ In this case the action on $V$
decomposes it into two submodules of dimension 1. If $r=2$ clearly
the action of $x$ is irreducible.
\end{proof}
Now observe that, from the embedding in $G=\Omega^{+}_{2n}(q)$ of
$$\Omega^-_{2m}(q)\bot \Omega^-_{2(n-m)}(q),$$
with $1\leq m< n/2\,$  we derive an element $\xi\in G$ of order
$$\,\frac{(q^{m}+1)(q^{n-m}+1)}{(q^{m}+1,q^{n-m}+1)(2,q-1)}.$$ If $(m,q)\neq (1,3)$ the action
of $\xi$ is of type $\,2m\oplus \,2(n-m)$ and otherwise, by Lemma
\ref{action},it is of type $1\oplus 1 \oplus 2(n-1).$

\begin{lemma}\cite[Theorem 1.1]{msw}\label {msw+}
Let $G=\Omega^+_{2n}(q)$, $n \ge 5.$ Let  $x\in G$  of order
$\,\frac{(q^{n-1}+1)(q+1)}{(q^{n-1}+1,q+1)(2,q-1)}\,$ and action of
type $2(n-1)\,\oplus\, 2$ if $q\neq 3$ and action of type
$2(n-1)\,\oplus\, 1\,
\oplus 1$ if $q=3.$\\
 If
$M$ is a maximal subgroup of $G$ containing $x$ then, up to
conjugacy, one of the following holds:
\begin{itemize}
\item[i)] $M=\Omega^+_{2n}(q) \cap(O^-_2( q)\bot\, O^-_{2(n-1)}(q));$
\item[ii)] $q=3$ and $M= \Omega_{2n-1}(3).2;$
  \item[iii)]
$n$ is even  and $M=\Omega^+_{2n}(q) \cap (GU_n(q^2).2); $
\item[iv)] $nq$ is odd and $M=\Omega_{n}(q^2).2;$

%\item[iv)]
%$(F^*(\overline{M}),G)=(\Omega_7(q),\Omega_8^+(q))$ and $q$ is odd;
%\item[v)] $(F^*(M),G)=(Sp_6(q),\Omega_8^+(q))$ and $q$ is even;
%\item[vi)] $(M,G)=(Alt(9),\Omega^+_{8}(2)).$
\end{itemize}
\end{lemma}
We emphasize that the maximal subgroups of $G$, containing $x$ are
obtained as a sublist of those containing a low-Singer
 cycle $\tilde{S}$ of order $\,\frac{(q^{n-1}+1)}{(2,q-1)}\,$
in \cite[Theorem 1.1]{msw}.

We now look for the  second component of a 2-covering.

\begin{lemma}\label{max-y}
Let $G=\Omega^{+}_{2n}(q),$ $n \ge 5$ and $y\in G$ of order
$\frac{(q^{n-2} +1)(q^2+1)}{( q^{n-2}+1,\,q^2 +1)(2,q-1)}$
 and action of type $2(n-2)\oplus
4.$ If $M$ is a maximal subgroup of $G$ containing $y$ then, up to
conjugacy, one of the following holds:
\begin{itemize}
\item[i)] $M=\Omega^+_{2n}(q) \cap(O^{-}_4(q) \bot\, O^{-}_{2(n -2)} (q))$;
\item[ii)] $n$ is even and $M=\Omega^+_n( q^2).[4].$
\end{itemize}
\end{lemma}
\begin{proof} Let $G=\Omega^{+}_{2n}(q)$, $n \ge 5$ and $M$ a maximal subgroup of
$G$ containing  $y,$ where $y\in G$ has order $\frac{(q^{n-2}
+1)(q^2+1)}{( q^{n-2}+1,q^2 +1)(2,q-1)}$ and action of type $2(n-2)\oplus 4.$\\
If $(n,q)= (5,2),$ the inspection in \cite{atlas} shows that the
only maximal subgroups of $\Omega^{+}_{10}(2)$ containing an element
of
order 45 are conjugate to $\Omega^+_{10}(2) \cap (O^{-}_4(2) \bot \,O^{-}_{6} (2)).$\\
If $(n,q)\neq (5,2),$ then $y$ is a strong $ppd(d,q;e)$-element, for
$d=2n$ and $e=2(n-2).$ By Theorem \ref{main}  $M$ belongs to one of
the classes $\mathcal C_i$, $i=1,\,2,\,3,\ 5$ or to $\mathcal{S}$
and is described in Example 2.6 a) of \cite{gpps}.

\vskip 3pt

\noindent $\mathcal{C}_1$.\ \ \  Suppose that $M$ is of type $P_m$.
Then, by Proposition 4.1.20 of \cite{kl}, the condition
$q_{2(n-2)}\mid |P_m|$ forces $m=1.$ On the other hand the
characteristic polynomial of $y$ is the product of two irreducible
factors of degree $4,\, 2(n-2);$ thus $y$ has no eigenvalues and it
cannot belong to a conjugate of $P_1.$\\
 Assume now  $M=\Omega^+_{2n}(q) \cap(O^{\epsilon}_m(q) \bot O^{\epsilon}_{2n-m}(q))$, with $ 1
\le m < n$ . Then $2n-m \ge 2n-4,$ hence $m \le 4.$ Since
$2(n-2)>4,$ the action of $y$ is compatible only with the choice
$m=4,\ \epsilon=-.$ \\ Finally if $M=Sp_{2(n-1)}( 2^f),$ then it is
the stabilizer of a non-degenerate subspace of dimension 1 and again
the action of $y$ excludes this opportunity.
%%%%%%%%%%%%%%%%%%%%%%%%%%%%%%%%%%%%%%%%%%
\vskip 3pt

\noindent $\mathcal{C}_2$.\ \ \ \ The maximal groups $M\in
\mathcal{C}_2$ are described in Example 2.3 of \cite{gpps}. Thus
$M \le GL(1,q) \wr Sym(2n) $ and $q_e= e +1= 2n -3$. Observe that
$5\leq q_4\neq q_e$ and $(q_e q_4, q-1)=1$: hence $q_e \cdot q_4$
is the order of an element in $Sym(2n)$ which implies $2n\geq
q_e+q_4\geq 2n+2,$ a contradiction.

%%%%%%%%%%%%%%%%%%%%%%%%%%%%%%%%%%%%%%%%%%
\vskip 3pt

\noindent $\mathcal{C}_3$.\ \ \ \ \ If $M$ is of type $GU_n(
q^2)$, $n$ even, then $n-2$ is even and  $q^{n-2} +1$ cannot
divide $|M|$.

If $M$ is of type $O_n( q^2)$, $n$ odd, then
$$\pi(|M|)=\pi\bigg(p \cdot \prod_{i=0}^{(n-1)/2}(q^{4i} - 1)\bigg
)
$$
 and the condition $q_{e}$ divides $|M|$, implies  $n-2 \mid 2i$
 that is $n-2$ divides $i \le (n-1)/2$, which is
impossible since $n \ge 5$.

We are left only with  $M=\Omega^+_{n}(q^2).[4]$, $n$ even.

%%%%%%%%%%%%%%%%%%%%%%%%%%%%%%%%%%%
\vskip 3pt

\noindent $\mathcal{C}_5$.\ \ \  By Remark \ref{no-c5}, the only
case  to consider is  $M$ of type $O ^-_{2n}( q_0)$, $q=q_0^2.$
Then
$$\pi(|M|) \subseteq \pi\bigg(p \cdot \prod_{i=0}^{n}(q^{i} -1)
\bigg),$$ and $q_e$ divides $|M|$ implies $2(n-2)\le i\le n$
against $n\geq 5.$

%%%%%%%%%%%%%%%%%%%%%%%%%%%%%%%
\vskip 4pt $\mathcal{S}$. \ \ \ The subgroup $M$ is such that
$Alt(m)\le M \le Sym(m) \times Z(G),$  with $m= 2n +1$, if $p$ does
not divide $m$, or $m=2n+2$ if $p$ divides $m$. Moreover $q_e= 2n -3
\ge m-5.$ Let $\sigma\in M,$ a power of $y$ of order $q_e.$ Since
$|Z(G)|=(4,q^n-1),$ it follows that we can assume
 $\sigma\in Sym(m)$ a $q_e-$cycle and
 $\frac{(q^2 +1)q_e}{(2,q-1)}$ is the order of an element in
$C_{Sym(m)}(\sigma)\times Z(G)= \langle \sigma \rangle \times
Sym(m-q_e)\times C_{(4,q^n-1)} \leq \langle \sigma \rangle \times
Sym(5)\times C_{4}.$ Thus $q^2 +1/(2,q-1)$ is the order of an
element in $Sym(5)\times C_{4},$ and thus $q=2, 3.$
\\Let $q=3.$ Then to fulfill the condition $2n-3$ prime which
divides $3^{n-2}+1,$ we need $n\geq 8.$ Moreover
$\frac{|y|}{(2n-3)}$ is the order of an element in $Sym(5)\times
C_{4}$ hence $\frac{(3^{n-2}+1)5}{(2n-3)(10,3^{n-2}+1) }\leq 24,$
that is $5(3^{n-2}+1)\leq 24(2n-3)(10,3^{n-2}+1),$ which gives no
solution for $n\geq 8.$\\ Let $q=2.$ If $n=5,$ then $2n-3=7$ does
not divide $2^3+1;$ if $n=,\ 6,\ 9$ then $2n-3$ is not a prime.  if
$n=7$ then $11=2n-3$ divides $2^5+1$ but $|y|/11=15$ is not the
order of an element in $Sym(5)\times C_4.$ If $n=8$ then $2n-3=13$
divides $2^6+1=65=|y|$ and $m=18.$ But it easily observed that
$Alt(18)\not\leq \Omega^+_{16}(2):$ namely in $Alt(18)$ there is an
element of order $11 \cdot 7$ but there is no element of such an
order in $\Omega^+_{16}(2).$ Observe also that if $n=10,$  the
condition $2_{16}=17$ is not realized. Finally, for $n\geq 11,$
there are no case since $\frac{(2^{n-2}+1)5}{(2n-3)(5,2^{n-2}+1)
}\leq 24$ cannot hold.
\end{proof}
%%%%%%%%%%%%%%%%
%%%%%%%%%%%%%%%%%%%

\begin{prop}\label{ortogonali +}
Let  $G=\Omega^+_{2n}(q)$, $n\ge 5$. Then $G$ is not 2-coverable.
\end{prop}
\begin{proof}
Let $G=\Omega^+_{2n}(q),$ $n \ge 5$ and suppose, by contradiction,
that  $\{H,K \}$ is a 2-covering of $G$. We can assume that $H$ is
described in Lemma \ref{msw+} and that $K$ is described in Lemma
\ref{max-y}.\\ We first assume $n$ odd. Then $$
H\in\{\Omega^+_{2n}(q) \cap(O^-_2( q)\bot\, O^-_{2(n-1)}(q)),\
\Omega_{n}(q^2).2 \}$$ or $q=3$ and $H=\Omega_{2n-1}(3).2,$ while
$$K=\Omega^+_{2n}(q) \cap(O^-_4( q)\bot\, O^-_{2(n-2)}(q)).$$ Let $g\in G$ with
$|g|=\frac{q^n-1}{(2,q-1)}.$ Then $|g|$ divides neither the order of
$H$ nor the order of $K.$\\ Next let $n$ be even. Then $$
H\in\{\Omega^+_{2n}(q) \cap(O^-_2( q)\bot\, O^-_{2(n-1)}(q)),\
\Omega^+_{2n}(q) \cap (GU_n(q^2).2)\}$$ or $q=3$ and
$H=\Omega_{2n-1}(3).2,$ while $$ K\in\{ \Omega^+_{2n}(q)
\cap(O^{-}_4(q) \bot\, O^{-}_{2(n -2)} (q)),\ \Omega^+_n(
q^2).[4]\}.$$ Let
 $g\in G$ with
$|g|=\frac{q^{n-1}-1}{(2,q-1)}.$ Note that $q_{n-1}$ does not divide
the order any candidates $H$ or $K,$ with the exception of $q=3$ and
$H=\Omega_{2n-1}(3).2.$ \\ So we can assume $n$ even, $q=3,$
$H=\Omega_{2n-1}(3).2 $ and $$ K\in\{ \Omega^+_{2n}(3)
\cap(O^{-}_4(3) \bot\, O^{-}_{2(n -2)} (3)),\ \Omega^+_n(9).[4]\}.$$
Pick in $G$ an element of order $\frac{3^{n/2}+1}{2}$ and action of
type $n\, \oplus\, n,$ which forces to $K$ to be $\Omega^+_n(9).[4]$
and finally observe that neither $H=\Omega_{2n-1}(3).2$ nor
$K=\Omega^+_n(9).[4]$ contain regular unipotent elements.

\end{proof}

%%%%%%%%%%%%%%%%%%%%%%%

\section{Orthogonal groups with Witt defect 1}

Let $G =\Omega^{-}_{2n}(q),\ n\geq 5$ and let $s\in G$ be a Singer
cycle. Then $|s|= \frac{q^{n}+1}{(2,q-1)}$ and the maximal subgroups
of $G$ containing $s$ are known.
\begin{lemma}\cite[Theorem 1.1]{msw}\label {msw-} If $M<\cdot\ \Omega^-_{2n}(q),\
 n\geq 5$
contains a Singer cycle then, up to conjugacy, one of the
following holds:
\begin{itemize}
  \item[i)]
$M= \Omega^-_{2n/r}(q^r).r,$ where $r$ is a prime divisor
of $n; $
 \item[ii)] $n$ is odd and $M= \Omega^-_{2n}(q)\cap (GU_{n}
 (q^2).2).$
\end{itemize}
\end{lemma}

%%%%%%%%%%%%%%%%%%%%%%%%%%%%%%%%%%%%
We now look for the  second component of a 2-covering.

\begin{lemma}\label{max-y-}
Let $G=\Omega^{-}_{2n}(q)$, $n \ge 5 .$ Let consider, for any $n \ge
5 ,$ an element  $y\in G$  with
$$|y|=\frac{(q^{n-1} +1)(q-1)}{(2,\,q-1)^2}$$  and for any $n$ even,
 an element $z\in
\Omega^{-}_{2n}(q)$ with $$|z|=\frac{(q^{n-1}
-1)(q+1)}{(2,\,q-1)^2},$$ action of type $(n-1)\oplus (n-1)\oplus 2$
if $q\neq 3$ and action of type $(n-1)\oplus (n-1)\oplus 1\oplus 1$
if $q=3.$
\\If $M<\cdot\ G$, up to conjugacy, contains $y $ and $z,$ then $n$
is odd and one of the following holds:

\begin{itemize}
\item[i)]  $M= P_1;$  %\in \mathcal{C}_1$;
 \item[ii)]  $q\geq 4,$  and
$M=\Omega_{2n}^-(q)\cap (O_2^+(q)\bot O_{2n-2}^-(q);$%\in \mathcal{C}_1, somma ortogonale con m=2;$
\item[iii)] $q=2,$  and $M=Sp_{2(n-1)}(2);$%\in \mathcal{C}_1, somma ortogonale con m=2;$
\item[iv)] $q$ is odd and $M=\Omega_n( q^2).2;$%\in \mathcal{C}_3;$
\end{itemize}
or
\begin{itemize}
\item[v)] $q=3,5$ and $M=\Omega_{2n}^-(q)\cap (O_1(q)\bot
O_{2n-1}(q)).$%\in \mathcal{C}_1 somma ortogonale con m=1;$
\end{itemize}
 \end{lemma}
\begin{proof}
Let $G=\Omega^{-}_{2n}(q),\  n \ge 5.$  Let $y\in
G$ be an element of order $ {\displaystyle \frac{(q^{n-1} +1)(q-1)}{(2,\,q-1)^2}}$ % and action
%$2(n-1)\oplus 1\oplus 1$
and for any $n$ even, $z\in \Omega^{-}_{2n}(q)$  be an element with
$|z|={\displaystyle\frac{(q^{n-1} -1)(q+1)}{(2,\,q-1)^2}}$ and
action of type $(n-1)\oplus (n-1)\oplus 2$ if $q\neq 3$ and action
of type $(n-1)\oplus (n-1)\oplus 1\oplus 1$ if $q=3.$
 Then for $d=2n$ and $e=2(n-1)$,
 $\,y\,$ is a strong $ppd(d,q;e)$-element of $GL_d(q).$\\
By Theorem \ref{main}, Remark  \ref{no-c5} and Table $3.5.$ F in
\cite{kl}, $M$ belongs to one of the classes $\,\mathcal C_i\,$,
$i=1,2,3$ or to $\mathcal S$ and it is described in the Examples
2.6-2.9 of \cite {gpps}.

\vskip 3pt

\noindent $\mathcal{C}_1$. \ \ \ If $M$ is of type $P_m,$ then, by
Proposition 4.1.20 of \cite{kl}, we have $n-m \ge n-1$, that is $m
=1.$ The choice $M=P_1$ is excluded when $n$ is even, since  $|z|$
does not divide $|P_1|.$
\\ Let $M=Sp_{2(n-1)}(q),$ with $q$ even. Since $Sp_{2(n-1)}(q)$ does not
contain a semisimple element of order a proper multiple of
$q^{n-1}+1,$ we get $q=2.$ But if $n$ is even, then $|z|$
 is  not the order of an element in $Sp_{2(n-1)}(2).$\\
 Let $M$ be of type $O_m^\epsilon(q) \bot O_{2n-m}^{-\epsilon}(q),$
with $1\leq m\leq n,$  $\,\epsilon \in \{+,\ -,\ \circ\}$ and $q$
odd when $m$ is odd. By the structure of $M$ given in Proposition
4.1.6 of \cite{kl}, we deduce that $q_e\,\mid \,|O_m^\epsilon(q)
\bot O_{2n-m}^{-\epsilon}(q)|,$ hence $m=2$ and $\epsilon=+$ or
$m=1$ and $\epsilon=\circ,$ $q$ odd. This gives  two possible
structures for $M.$ \\The first, related to $m=2$, is
$$M=\Omega_{2n}^-(q)\cap(O_2^+(q)\bot O_{2n-2}^-(q)),$$  with $q\not \in\{2,\ 3\}$
it is the natural component for $y.$ When $n$ is even, this subgroup
cannot contain $z.$\\By Table $3.5.$ H in \cite{kl}, for $q=2$ we
get again $M=Sp_{2(n-1)}(2)$ and for $q=3$ we have
$M=\Omega_{2n}^-(3)\cap(O_1(3)\bot O_{2n-1}(3)),$ which contain a
conjugate of $z.$\\
The second structure for $M$, related to $m=1,$ is
$$M=\Omega_{2n-1}(q).c$$ with $c\mid 2$ and $q\neq 3$ odd. The group $\Omega_{2n-1}(q)$ does not contain
semisimple elements with order properly divisible by $\frac{q^{n-1}
+1}{2},$ hence we get the condition $\frac{q-1}{2}\leq 2,$ which
leaves us with $q= 5.$

%%%%%%%%%%%%%%%%%%%%%%%%%%%%%
\vskip 3pt

\noindent $\mathcal{C}_2$.\ \ \ Let $M\in \mathcal{C}_2.$ Then $M$
is described in Example 2.3 in \cite{gpps}, $q_e= e +1= 2n -1$ and,
by Proposition 4.2.15 in \cite{kl}, $q\equiv 3(mod\ 4)$ is a prime,
$n$ is odd and $M\leq 2^{2n}.Sym(2n).$ Thus we get $q^{n-1}+1\equiv
3^{n-1}+1(mod\ 4)\equiv 2(mod\  4).$ On the other hand
$(q^{n-1}+1)/2>2n-1;$ then we obtain an odd prime $r$ with
$r\,q_e\mid |y|$ and an element of order $r\,q_e$ in $Sym(2n).$
Therefore $2n\geq r+q_e\geq 2n+2,$ a contradiction.

%%%%%%%%%%%%%%%%%%%%%%%%%%%%%%%
\vskip 3pt \noindent  $\mathcal{C}_3$. \ \ \ If $M$ is of type
$GU_n( q^2)$, $n$ odd, then $q_{n-1} $ does not divide $|M|$.

If $M=\Omega_n(q^2).2$, with $qn$ odd, then there is an element of
the required order.

If $M=\Omega_{2n/r}^-(q^r).r$, with $r$ a prime dividing $n$, then
$q_e\neq r$ does not divide $|M|$.

%%%%%%%%%%%%%%%%%%%%%%%%%%%%%%%
%%%%%%%%%%%%%%%%%%%%%%%%%%%%%%%
\vskip 3pt \noindent  $\mathcal{S}$. \  We examine the Examples
2.6-2.9 in which $\mathcal{S}$ decomposes. Observe that we have
$e=d-2$ and that, since $d\geq 10,$ we exclude Examples 2.6 {\it(b),
(c)}.

\vskip 3pt \noindent{\it Example 2.6(a): S an Alternating Group.}\\
 Here we have $ Alt(m)\le M \le
Sym(m) \times Z(G),$ with $\,m-1= 2n\,,$ if $\,p\,$ does not divide
$\,m\,$ or $\,m-2=2n\,$ if $\,p\,$ divides $\,m\,$. Moreover $q_e=
e+1= 2n -
1\geq  m-3$ which gives also $q_e \ge 7.$\\
Let $\sigma$ be an element of $Sym(m)$ of order $q_e$. Then $\sigma$
is a $q_e-$cycle and $|y|$ divides the order of $|C_{ Sym(m) \times
Z(G)}(\sigma)|$, which implies $q^{n-1}+1$
%\le 12(2n-1),$
 divides $ 24(2n-1).$ The condition $2n - 1$ prime gives $n\geq 6.$ \\
Let $n=6$ and $q=2;$ then $m\in \{13,14\}$ and $|y|=33$ requires
$m=14,$ which leaves us with the case $ Alt(14)\le M \le Sym(14).$
But then $\Omega_{12}^-(2)$ would be a minimal module for $Alt(14),$
against the fact that $Alt(14)$ fixes no quadratic form on its
minimal module (see Proposition 5.3.7 in \cite{kl}).
\\
If $n\geq 7$ and $q=2$ or if $n\geq 6$ and $q>2,$ the condition
$q^{n-1}+1$ divides $ 24(2n-1)$ cannot hold.
%%%%%%%%%%%%%%%%%%%%%%%%%%%%%%

\vskip 3pt \noindent {\it Example 2.7: $S$ a sporadic simple
group.}\\
Recall that the centralizer of an element of a sporadic group can
be easily checked in \cite{atlas}.
 There are five cases with $e=d-2$ and $d\geq 8$ even
in Table 5. Observe that in the first column of that Table, we read
$M'=C.S,$ where $C$ embeds in $Z(G).$ In particular $|C|$ divides
 $4.$ This reduces our analysis only to three cases.
 Observe also that $M\leq C.Aut(S).$\\
 If $n=6$ and $M'=2.M_{12}$, then
$q_e=11$ and since the centralizer in $2.Aut(M_{12})\geq M$ of an
element of order 11 has order $22$, this implies $\frac{(q^5
+1)(q-1)}{(2,q-1)^2} \le 22$,
a contradiction.\\
If $n=10$ and $M'=J_1=M$, $q_e=19$ and $|C_S(g)| =19$, for any $g$
of order 19 in $S$. This implies $\frac{(q^9 +1)(q-1)}{(2,q-1)^2}
\le 19$, a
 contradiction.\\
If $n=12$ and $M'= 2.Co_1$, $q_e=23$, then $|C_{2.Co_1}(g)|= 46$,
for any $g$  element of order 23 in $S$ and we get $q^{11} +1 \le
46\cdot 2$, a contradiction.

%%%%%%%%%%%%%%%%%%%%%%%%%%%%%%%%%%%%%%%%%

\vskip 3pt \noindent {\it Example 2.8: S  a simple group of Lie type
in characteristic $p.$}\\ No case arises.
%%%%%%%%%%%%%%%%%%%%%%%%%%%%%%%%%%%%%%%%%

\vskip 3pt \noindent {\it Example 2.9: S a simple group of Lie
type in characteristic different from $p.$}\\
In Table 7, we have the examples with $n=7$, $q_e=13,$
 and $\frac{q^6 +1}{(2,q-1)}$ dividing either $8 \cdot 13$ or $12 \cdot 13,$
 which never happens.
We also have the case $n=9$, $q_e=17$ and $\frac{q^8 +1}{2}$
dividing
 $17\cdot  4$, which never happens.
In Table 8, we have to consider the cases in which $e=d-2$, that is
$M/Z(M)$ is isomorphic to a subgroup of $Aut(S)$ containing $S$ and
either $S\cong PSp_{2a}(3)$, for some odd prime $a$, or $S \cong
PSL_2(s),$ for some $s \ge 7$.\\ In the first case $d=2n = (3^a
+1)/2$, $e= 3(3^{a-1} -1)/2= 2(n-1)$, which implies $3$ divides
$n-1.$ The order of $y$ must divide the order of  the centralizer in
$M$ of an element of order $q_e=(3^a -1)/2$ and clearly $M\cong
b.PSp_{2a}(3).c$ where $b,\ c\mid 2.$ Since $(3^a -1)/2$ is the
order of a maximal cyclic torus in $PSp_{2a}(3),$ this implies $|y|$
divides $4 \cdot q_e$ hence $\frac{q^{(n-1)/3} +1}{(2,q-1)}$ divides
4, thus $q^2\leq 7,$ which gives $q=2,$ but $2^{(n-1)/3}+1$ does not
divide 4.

 If $S \cong PSL_2(s),$ $s \ge
7$, we examine the various subcases. If $2n=s=2^c$, with $c$ prime
then $q_e= s-1$, and $M$ embeds in $SL_2(s).c,$ where the cyclic
extension of order $c$ is conjugate to a field automorphism. Now
observe that, since $q_e\neq c,$ an element $g$ of order $q_e$
 in $SL_2(s).c$ is a cyclic torus in
  $SL_2(s)$ and that a field automorphism does not centralize it. Thus we get
$|C_M(g)|\leq q_e$. This implies that $|y|=q_e$. But $n-1$ is odd,
and therefore $q^{n-1} +1$ is
divisible by $(q+1) \cdot q_e$, a contradiction.\\
If  $d=s+1= 2n,$ $s$ an odd prime and $q_e=s= e+1= 2n-1$ we have
$M\leq SL_2(s).2.$ Let $g$ be an element  of order $s$ in $M.$
Then $g^2$ is an element of order $s$ in $M$ and
$|C_{M}(g)|=|C_{M}(g^2)|$ divides  $4s.$ Thus we obtain the
condition
   $\frac{(q^{n-1}+1)(q-1)}{(2,q-1)^2}$ divides $4(2n-1)$, which
 has no solution.
 The only case still to examine is $d=(s+1)/2=
2n$, $q_e=(s-1)/2=2n-1$, $s=r^f,\ r$ odd. Observe that here $n\geq
6$  and $M\leq SL_2(s).2f.$ Let $g\in M$ of order $q_e;$ since $q_e$
cannot divide $f,$ then  $g^{2f}$ has order $q_e=(s-1)/2$ and
belongs to $SL_2(s).$ Due to the partition of $PSL_2(s),$ it is
clear that $g^{2f}$ is conjugate to the diagonal matrix $D$ with
diagonal entries $\lambda^2,\ \lambda^{-2},$ where
$<\lambda>=\mathbf F_s^*$ and that $|C_{SL_2(s)}(g^{2f})|=s-1.$
 Moreover, a field automorphism does not centralize $D,$ hence
  $\,|C_M(g)|=|C_M(g^{2f})|$ divides
  $2|C_{SL_2(s)}(g^{2f})|=2(s-1)=4(2n-1).$
  This  implies, as before, the impossible relation $|y|$ divides $4(2n-1).$
  \end{proof}

\begin{prop}\label{ortogonali-}
Let  $G=\Omega^-_{2n}(q)$, $n\ge 5$. Then $G$ is not 2-coverable.
\end{prop}
\begin{proof} Let $\delta=\{H,K \}$ be a 2-covering of
$G=\Omega^-_{2n}(q).$ We can assume
       $H$ as described in Lemma
\ref{msw-} and $K$ as described in Lemma \ref{max-y-}. % since there
%is no overlap between the lists of maximal subgroups of $G$ in
%these lemmata\\
When $q=3, 5$ we observe that $K=\Omega_{2n}^-(q)\cap (O_1(q)\bot
O_{2n-1}(q))$ does not contain regular unipotent elements and that
no candidate $H$ can contain them. Thus, by Lemma \ref{max-y-},
 we reduce our attention to $n$ odd and $K\neq \Omega_{2n}^-(q)\cap (O_1(q)\bot
O_{2n-1}(q)).$ Let $u$ be an element of order $\,\frac{(q^{n-2}
-1)(q^2 +1)}{(q^{n-2}-1, q^2+1)(2,q-1)}\,$ belonging to
$\,\Omega_4^-(q)\bot \Omega_{2(n-2)}^+(q)<G$ and with action of type
$\,4\oplus (n-2)\oplus (n-2)\, .$ Then $u$ has no eigenvalues and
its order is divisible by $q_{n-2}$. But the matrices in any
candidate $K,$ except $\Omega_n( q^2).2,$ admit an eigenvalue and
$q_{n-2}$ does
 not divide the order of $\Omega_n( q^2).2,\ GU_{n}(q^2).2,\
 \Omega^-_{2n/r}(q^r).r.$
\end{proof}

%%%%%%%%%%%%%%%%%%%%%%%%%%%%%

%%%%%%%%%%%%%%%%%%%%%%%%%%%%%%%%%%%%%%%%%%%%%%%%


\begin{thebibliography}{99}
\bibitem{a} M. Aschbacher, {\em On the maximal subgroups of the finite
classical groups}, Invent. Math. \textbf{76}(1984), 469-514.


\bibitem{bl} D. Bubboloni, M. S. Lucido,
       {\em Coverings of linear groups},
Comm. Algebra, n. {\bf 30 (5)} (2002), 2143-2159.
\bibitem{blw} D. Bubboloni, M. S. Lucido and T. Weigel, {\em
Generic 2-coverings of finite groups of Lie-type}, Rend. Sem. Mat.
Padova, vol. \textbf{115} (2006), 209-252.
\bibitem {atlas} J. Conway, R. Curtis, S. Norton, R. Parker and R.
Wilson, Atlas of finite Groups, Clarendon Press, Oxford, 1985.


\bibitem{Dye} R. H. Dye, {\em Interrelations of Symplectic and
Orthogonal Groups in Characteristic Two}, J. Algebra \textbf{59}
(1979), 202-221.
   \bibitem{GLS3} D. Gorenstein, R. Lyons, R. Solomon.  The
      classification of the finite simple groups, Number 3.  Amer.
Math.
      Soc.  Surveys and Monographs {\bf 40}, 3 (1998).
\bibitem{gpps} R. Guralnick, T. Penttila, C. E. Praeger and J. Saxl,
          {\em Linear groups with orders having certain large prime
divisors}, Proc. London Math. Soc. (3), n.{\bf 78 (1)} (1999),
167-214.
\bibitem{hu} B. Huppert, {\em Singer-Zyklen in Klassischen
Gruppen}, Math.Z. {\bf 117}(1970), 141-150.

\bibitem{kl} P. B. Kleidman and M. W. Liebeck, The subgroup
structure of the finite classical groups, London Math. Soc.
Lecture Notes {\bf 129}, Cambridge University Press, 1990.

\bibitem{msw} G. Malle, J. Saxl and T. Weigel,  {\em Generation of
classical groups,} Geometriae Dedicata {\bf 49}(1994), 85-116.



\bibitem{zs} K. Zsigmondy, {\em Zur Theorie der Potenzreste},
Monathsh. Fur Math. u. Phys.\textbf{3} (1892), 265-284.
\end{thebibliography}
\end{document}